\documentclass[10pt]{siamltex}
\pdfoutput=1
\usepackage{a4,amsmath,amsfonts,graphicx}

\newtheorem{hypothesis}{Hypothesis}
\usepackage{url,color}
\usepackage[pdftex,colorlinks]{hyperref}
\definecolor{darkblue}{cmyk}{1,0,0,0.8}
\definecolor{darkred}{cmyk}{0,1,0,0.7}
\hypersetup{anchorcolor=black,
  citecolor=darkblue, filecolor=darkblue,
  menucolor=darkblue,pagecolor=darkblue,urlcolor=darkblue,linkcolor=darkblue}
\usepackage[scaled=0.9]{helvet}
\usepackage[charter]{mathdesign}
\usepackage[font=small,up]{caption}
\usepackage{microtype} 
\usepackage{mythmbox}

\newtheorem[S]{corollary}{Corollary}[section]
\newtheorem[S]{theorem}[corollary]{Theorem}
\newtheorem[S]{definition}[corollary]{Definition}
\newtheorem[S]{hypothesis}[corollary]{Hypothesis}
\newtheorem[S]{lemma}[corollary]{Lemma}
\newtheorem[S]{proposition}[corollary]{Proposition}
\newtheorem[S]{introthm}{Main Result}

\newcommand{\R}{\mathbb{R}}
\newcommand{\C}{\mathbb{C}}
\newcommand{\Lint}{\mathbb{L}}
\newcommand{\id}{I}

\renewcommand{\d}{{\mathrm{d}}}
\renewcommand{\phi}{\varphi}

\title{Characteristic matrices for linear periodic delay differential
  equations}
\author{Jan Sieber\footnotemark[1] \and  Robert Szalai\footnotemark[2]}
\begin{document}
\maketitle
\footnotetext[1]{Dept. of Mathematics, university of Portsmouth (UK)}
\footnotetext[2]{Dept. of Engineering Mathematics, University of
  Bristol (UK)}
\begin{abstract}
  \noindent Szalai et al.\ (SIAM J.\ on Sci.\ Comp.\ 28(4), 2006) gave
  a general construction for characteristic matrices for systems of
  linear delay-differential equations with periodic
  coefficients. First, we show that matrices constructed in this way
  can have a discrete set of poles in the complex plane, which may
  possibly obstruct their use when determining the stability of the
  linear system. Then we modify and generalize the original
  construction such that the poles get pushed into a small
  neighborhood of the origin of the complex plane.
\end{abstract}
\begin{AMS}
  34K06, 34K08, 34K20
\end{AMS}
\begin{keywords}
  delay differential equations, characteristic matrix, stability of
  periodic orbits
\end{keywords}
\section{Introduction}
\label{sec:intro}
Linear delay differential equations (DDEs) with coefficients periodic
in time come up naturally when one analyses nonlinear DDEs: one of the
most common and simplest invariant sets encountered in nonlinear
systems are periodic orbits \cite{GH83}. If one wants to find out if a
periodic orbit is dynamically stable or unstable, or how its stability
changes as one changes system parameters one has to linearize the
nonlinear DDE system in the periodic orbit, which results in a linear
DDE with time-periodic coefficients \cite{DGLW95,HL93}.  This linear
DDE is typically written in the form
\begin{equation}
  \label{eq:introdde}
  \dot x(t)=L(t)x_t
\end{equation}
where the coefficients (hidden in the operator $L(t)$) are periodic in
$t$. Without loss of generality we can assume that the time dependence
of $L$ is of period $1$ (this corresponds to rescaling time and $L$).
Then the exponential asymptotic stability of the periodic orbit in the
full nonlinear system is determined by the eigenvalues of the time-$1$
map $T$ generated by the linear DDE \eqref{eq:introdde}. See the
textbooks \cite{DGLW95,HL93} for a summary of the fundamental results
on Floquet theory and \cite{RS07} for robust numerical methods to
calculate these eigenvalues in practical problems. For analytical
purposes (and, possibly, for numerical purposes) it is useful to
reduce the eigenvalue problem for the infinite-dimensional time-$1$
map $T$ to an algebraic problem via a characteristic matrix
$\Delta(\lambda)\in\C^{n\times n}$. One would expect that this matrix
$\Delta$ should satisfy, for example, that
\begin{romannum}
\item $\Delta(\lambda)$ is regular if (and only if) $\lambda\in\C$ is
  in the resolvent set of $T$ (that is, $\lambda\id-T$ is an
  isomorphism),
\item $\lambda$ is a root of $\det\Delta(\cdot)$ of order $q$ if and
  only if $\lambda$ is an eigenvalue of $T$ of algebraic multiplicity
  $q$ (see \cite{KL92} or Lemma~\ref{thm:multiplicity} in
  \S\ref{sec:fix} how one can construct the Jordan chains of $T$
  from $\Delta$), and
\item $\dim\ker\Delta(\lambda)$ is the geometric multiplicity of $\lambda$
  as an eigenvalue of $T$.
\end{romannum}
An immediate consequence of the existence of such a matrix $\Delta$
would be the upper limit $n$ (the dimension of $\Delta(\lambda)$) on the
geometric multiplicity of eigenvalues of $T$. 

For time-invariant linear DDEs (where $L$ does not depend on $t$ in
\eqref{eq:introdde}) the existence of a characteristic matrix has been
known and used for a long time. The general theoretical background and
construction (extended to more general evolutionary systems such as,
for example, neutral equations and some classes of hyperbolic partial
differential equations) is given by Kaashoek and Verduyn-Lunel in
\cite{KL92}.

For time-periodic linear DDEs the textbook of Hale and Verduyn-Lunel
\cite{HL93} has developed characteristic matrices only for the case
where all time delays are multiples of the period. This has been
generalized to delays depending rationally on the period in order to
derive analytical stability results for periodic orbits of the
classical scalar delayed positive feedback problem \cite{SW06}. A
characteristic function in the form of a continued fraction expansion
is given for a scalar DDE in \cite{J00}. In general, precise
analytical statements about the location of Floquet multipliers (or
the number of unstable Floquet multipliers) can be made only for DDEs
with very special structure, typically scalar DDEs with single delays
and particular requirements on the right-hand side, see
\cite{CW88,D92,W91,X92} for classical works. An extension of these
results to cyclic DDEs using a discrete Lyapunov function is given in
\cite{MS96}.

Szalai et al.\ give a general construction for a characteristic matrix
$\Delta$ for time-periodic DDEs in \cite{SSH06} that lends itself
easily to robust numerical computation. Since one expects that
$\Delta(\lambda)$ must have an essential singularity at $\lambda=0$
(the eigenvalues of the compact time-$1$ map $T$ accumulate at $0$)
the constructed matrix is rather of the form $\Delta(\mu)$ where
$\mu=\lambda^{-1}$. We show (in \S\ref{sec:problem}) that the
characteristic matrix $\Delta(\mu)$ as constructed by Szalai et al.\
typically has poles in the complex plane, making it unusable whenever
eigenvalues of $T$ of interest coincide with these poles. However, the
set of these poles is discrete and accumulates only at $\infty$.
Moreover, in \S\ref{sec:fix} we provide a modification $\Delta_k(\mu)$
of the original matrix $\Delta(\mu)$, which allows us to push the
poles of $\Delta_k$ to the outside of any given ball of radius
$R\geq1$ in the complex plane by increasing the dimension of
$\Delta_k$. This modification reduces the eigenvalue problem for $T$
to a root-finding problem for $\det\Delta_k(\mu)$ for all eigenvalues
of $T$ with modulus larger than $1/R$, which is useful because one is
typically interested in the largest eigenvalues of $T$. To keep the
notational overhead limited we develop our arguments for the case of a
DDE with a single delayed argument and a constant delay $\tau<1$. In
Appendix~\ref{sec:gen} we show that the generalization to arbitrary
delays (distributed and reaching arbitrarily far into the past) is
straightforward.

One immediate application for the characteristic matrix came up in
\cite{YP09} where Yanchuk and Perlikowski consider periodic orbits of
nonlinear DDEs with a single fixed delay $\tau$ but change the delay
to $\tau+kP$ (where $k$ is an integer and $P$ is the period of the
periodic orbit), and then study the stability of the periodic orbit
(which does not change its shape) in the limit $k\to\infty$. The limit
is, of course, a singular limit if one considers the time-$P$ map $T$
but the characteristic matrix $\Delta_k$ has a well-defined regular
limit. This permitted the authors to draw conclusions about the linear
stability properties of periodic orbits for sufficiently large delays
from the properties of the so-called \emph{pseudo-continuous spectrum}
\cite{YP09}. As \cite{YP09} relies on the original construction from
\cite{SSH06}, which may have poles near the unit circle, our note
closes a gap in the argument of \cite{YP09}.

\section{Construction of characteristic matrix 
  for a single fixed delay}
\label{sec:constr}
Consider a periodic linear differential equation of dimension $n$ with
a single delay $\tau$ and continuous periodic coefficient matrices
$A(t)$ and $B(t)$ (both of size $n\times n$):
\begin{equation}
  \label{eq:spdde}
  \dot x(t)=A(t)x(t)+B(t)x(t-\tau)
\end{equation}
where we assume that the period of the time dependence of $A$ and $B$
equals $1$ without loss of generality (this corresponds to rescaling
time, $\tau$, $A$ and $B$). We also assume for simplicity of notation
that $\tau<1$.  Let us denote the monodromy operator (also called the
time-$1$ map) for \eqref{eq:spdde} by $T: C([-1,0];\C^n)\mapsto
C([-1,0];\C^n)$. $C([-1,0];\C^n)$ is the space of continuous functions
on the interval $[-1,0]$ with values in $\C^n$. That is, $T$ maps an
initial condition $x$ in $C([-1,0];\C^n)$ to the solution of
\eqref{eq:spdde} after time $1$ starting from the initial value $x(0)$
(the head point) and using the history segment $x(\cdot)$. For $x\in
C([-1,0];\C^n)$ we define $Tx$ precisely as the solution $y\in
C([-1,0];\C^n)$ of
\begin{equation}\label{eq:tdef}
  \begin{split}
    \dot y(t) &=A(t) y(t)+B(t)
    \begin{cases}
      y(t-\tau)&\mbox{if  $t\in[\tau-1,0]$}\\
      x(1+t-\tau) &\mbox{if $t\in[-1,\tau-1)$}
    \end{cases}\\
    y(-1)&=x(0)\mbox{.}
  \end{split}
\end{equation}
The complex number $\lambda$ is an eigenvalue of $T$ (a \emph{Floquet
  multiplier} for \eqref{eq:spdde}) if one can find a non-zero $x$
such that $Tx=\lambda x$.  We know that $T$ is compact \cite{HL93} for
$\tau<1$ (and that, for arbitrary $\tau$, we can find a power $T^m$
that is compact).  Hence, any non-zero $\lambda$ is either in the
resolvent set of $T$ (that is, $\lambda\id-T$ is regular) or it is an
eigenvalue of $T$ with a finite algebraic multiplicity.

If we introduce $\mu=1/\lambda$ we may write the eigenvalue problem
for $T$ as follows:
\begin{align}
  \label{eq:bvp}
  \dot x(t)&=A(t)x(t)+B(t)
  \begin{cases}
    x(t-\tau) &\mbox{if $t\in[\tau-1,0]$,}\\
    \mu x(1+t-\tau) &\mbox{if  $t\in[-1,\tau-1)$,}
  \end{cases}\\
  \label{eq:bc}
  x(-1)&=\mu x(0)
\end{align}
where $x\in C([-1,0];\C^n)$ is continuous.  Reference \cite{SSH06}
constructs a characteristic matrix $\Delta(\mu)$ for $T$ based on the
following hypothesis.
\begin{hypothesis}[Szalai'06]\label{thm:sz}
The initial-value problem on the interval $[-1,0]$
\begin{align}
  \label{eq:ivp}
  \dot x(t)&=A(t)x(t)+B(t)
  \begin{cases}
    x(t-\tau) &\mbox{if $t\in[\tau-1,0]$,}\\
    \mu x(1+t-\tau) &\mbox{if  $t\in[-1,\tau-1)$,}
  \end{cases}\\
  \label{eq:ic}
  x(-1)&=v
\end{align}
has a unique solution $x\in C([-1,0];\C^n)$ for all $\mu\in\C$ and
$v\in\C^n$.
\end{hypothesis}
Note that the problem \eqref{eq:ivp}--\eqref{eq:ic} differs from the
boundary-value problem \eqref{eq:bvp}--\eqref{eq:bc} because we have
replaced the boundary condition $x(-1)=\mu x(0)$ by an initial
condition $x(-1)=v$. The formulation of the hypothesis in \cite{SSH06}
is (only apparently) weaker: it claims the existence only for
$\mu\in\C$ for which $1/\mu$ is either in the resolvent set of $T$ or
an eigenvalue of $T$ of algebraic multiplicity one.  If
Hypothesis~\ref{thm:sz} is true then one can define the characteristic
matrix $\Delta(\mu)$ via
\begin{equation}
  \label{eq:delta}
  \Delta(\mu)v=v-\mu x(0)
\end{equation}
where $x(0)$ is the value of the unique solution of
\eqref{eq:ivp}--\eqref{eq:ic} at the end of the time interval
$[-1,0]$.

\paragraph*{Integral equation for \eqref{eq:ivp}--\eqref{eq:ic}}
For system~\eqref{eq:ivp}--\eqref{eq:ic} one has to clarify what it
means for $x\in C([-1,0];\C^n)$ to be a solution. We call $x\in
C([-1,0];\C^n)$ a solution of \eqref{eq:ivp}--\eqref{eq:ic} if $x$
satisfies the integral equation
\begin{align}
  \label{eq:varp}
  x(t)&=v+ M_1(\mu)[x](t)\quad\mbox{where}\\
  \label{eq:varm}
  \left[M_1(\mu)x\right](t)&=\int_{-1}^tA(s)x(s)+B(s)
    \begin{cases}
      x(s-\tau) &\mbox{if $s\in[\tau-1,0]$,}\\
      \mu x(1+s-\tau) &\mbox{if  $s\in[-1,\tau-1)$}
    \end{cases}\d s
\end{align}
pointwise for all times $t$ on the interval $[-1,0]$. The linear operator
$M_1(\mu)$ maps $C([-1,0];\C^n)$ back into $C([-1,0];\C^n)$ and is
compact. ($M_1$ will be generalized later to $M_k$ for $k>1$.)

\begin{lemma}\label{thm:smallmu}
  For $\mu$ sufficiently close to $0$ the initial-value problem
  \eqref{eq:ivp}--\eqref{eq:ic} has a unique solution for all $v$.
\end{lemma}
\begin{proof}
  For \eqref{eq:varp} to have a unique solution it is enough to show
  that $\id-M_1(\mu)$ is invertible. Since the time-$1$ map $T$ is
  well defined and linear in $x=0$ (that is, $Tx=0$ for $x=0$) we have
  that \eqref{eq:tdef} has only the trivial solution for $x=0$. This
  in turn implies that the corresponding integral equation
  \begin{equation}\label{eq:fixp0}
    y=M_1(0)y
  \end{equation}
  has only the trivial solution in $C([-1,0];\C^n)$. Thus, the kernel
  of $\id-M_1(0)$ is trivial, and, since $\id-M_1(0)$ is a compact
  perturbation of the identity, $\id-M_1(0)$ is invertible. Moreover,
  the real-valued function $\mu\mapsto\|M_1(\mu)-M_1(0)\|$ (using the
  operator norm induced by the maximum norm of $C([-1,0];\C^n)$) is
  continuous in $\mu$ because $M_1$ depends continuously on $\mu$ (see
  definition of $M_1$ in \eqref{eq:varm}). This implies that
  $\id-M_1(\mu)$ is also invertible for small $\mu$, which in turn
  means that \eqref{eq:varp}--\eqref{eq:varm} has a unique solution
  for small $\mu$.
\end{proof}

The following lemma states that the function on the complex domain
$\mu\mapsto[\id-M_1(\mu)]^{-1}$, which has operators in
${\cal L}(C([-1,0];\C^n);C([-1,0];\C^n))$ as its values, is well defined and
analytic everywhere in the complex plane except, possibly, in a
discrete set of poles.
\begin{lemma}\label{thm:finpoles}
  The operator-valued complex function $\mu\mapsto[\id-M_1(\mu)]^{-1}$
  can have at most a finite number of poles in any bounded subset of
  $\C$. All poles of $[\id-M_1(\mu)]^{-1}$ (if there are any) are of
  finite multiplicity. If $\mu$ is not a pole then
  $\mu\mapsto[\id-M_1(\mu)]^{-1}$ is analytic in $\mu$.
\end{lemma}
\begin{proof}
  We can split the operator $M_1(\mu)$ into a sum of two operators:
  \begin{align}
    M_1(\mu)&= M_1(0)+\mu L \mbox{\quad where}\label{eq:msplit}\\
    \left[Lx\right](t)&=\int_{-1}^tB(s)
    \begin{cases}
      0 &\mbox{if $s\in[\tau-1,0]$,}\\
      x(1+s-\tau) &\mbox{if  $s\in[-1,\tau-1)$}
    \end{cases}\d s\mbox{.}\label{eq:ldef}
  \end{align}
  The operator $\id-M_1(0)$ is invertible (and a compact perturbation of
  the identity) and $L:C([-1,0];\C^n)\mapsto C([-1,0];\C^n)$ is compact.
  Thus, we can rewrite $\id-M_1(\mu)$ as
  \begin{align*}
    \id-M_1(\mu)\phantom{]^{-1}}&=\left(\id-M_1(0)\right)
    \left[\id-\mu(\id-M_1(0))^{-1}L\right]&&\mbox{and, thus,}\\
    [\id-M_1(\mu)]^{-1}&=
    \lambda\cdot\left[\lambda\id-(\id-M_1(0))^{-1}
        L\right]^{-1}(\id-M_1(0))^{-1}
  \end{align*}
  (keeping in mind that $\mu=1/\lambda$). The function
  $\lambda\mapsto[\lambda\id-(\id-M_1(0))^{-1}L]^{-1}$ on the
  right-hand side is the standard resolvent of the compact operator
  $(\id-M_1(0))^{-1}L$, which has only finitely many poles of finite
  multiplicity outside of any neighborhood of the complex origin
  according to \cite{K66}, and is analytic everywhere else.
\end{proof}

Lemma~\ref{thm:finpoles} carries over to the characteristic matrix
$\Delta(\mu)$ defined in \eqref{eq:delta}: since
$x=[\id-M_1(\mu)]^{-1}v$ (if we extend $v\in\C^n$ to a constant
function), $\Delta(\mu)v=v-\mu x(0)$  is an analytic function with
at most a discrete set of poles of finite multiplicity (possibly
accumulating at $\infty$).

The characteristic matrix $\Delta(\mu)$ (where it is well defined) has
the properties one expects: for a complex number $\mu$ in the domain of
definition of $\Delta$, $1/\mu$ is a Floquet multiplier of the DDE
\eqref{eq:spdde} if and only if $\Delta(\mu)$ is singular
($\det\Delta(\mu)=0$).  Any vector $v$ in its kernel is the value of
an eigenfunction for $1/\mu$ at time $t=-1$. The full eigenfunction
can be obtained as the solution of the initial-value problem
\eqref{eq:ivp}--\eqref{eq:ic}.  If $\Delta(\mu)$ is regular then
$1/\mu$ is in the resolvent set of the eigenvalue problem
\eqref{eq:bvp}--\eqref{eq:bc}. The Jordan chains associated to
eigenvalues $1/\mu$ can be obtained by following the general theory
described in \cite{KL92}. We give a precise statement (see
Lemma~\ref{thm:multiplicity} in Section~\ref{sec:fix}) about the
Jordan chain structure after discussing Hypothesis~\ref{thm:sz}.

The construction can be extended in a straightforward manner to
multiple fixed discrete delays and distributed delays. The only
modification is that higher powers of $\mu$ have to be included for
delays larger than $1$. For example, if $\tau\in(1,2]$ then the
initial-value problem \eqref{eq:ivp}--\eqref{eq:ic} is modified to
\begin{align*}
  \dot x(t)&=A(t)x(t)+B(t)
  \begin{cases}
    \mu x(1+t-\tau) &\mbox{if $t\in[\tau-2,0]$,}\\
    \mu^2 x(2+t-\tau) &\mbox{if  $t\in[-1,\tau-2)$,}
  \end{cases}\\
  x(-1)&=v\mbox{.}
\end{align*}
Thus, also for arbitrary delays $\Delta(\mu)$ always has at worst
finitely many poles of finite multiplicity in any bounded subset of
$\C$.  Appendix~\ref{sec:gen} discusses the case where the functional
accessing the history of $x$ is a Lebesgue-Stieltjes integral.

\section{The poles of the characteristic matrix}
\label{sec:problem}
Hypothesis~\ref{thm:sz} is not true in its stated form (and neither in
the form stated in \cite{SSH06}). 
Let us analyse the simple scalar example
\begin{equation}
  \label{eq:cdde}
  \dot x(t)=x\left(t-\frac{1}{2}\right)
\end{equation}
on the interval $[-1,0]$. Problem \eqref{eq:ivp}--\eqref{eq:ic}
restated for this example is equivalent to
\begin{equation}
  \label{eq:exdde}
  \begin{aligned}
    \dot x_1(t)&=\mu x_2(t)\mbox{,} & x_1(0)&=v\\
    \dot x_2(t)&=x_1(t)\mbox{,} & x_2(0)&=x_1(1/2)
  \end{aligned}
\end{equation}
where $x_1$ and $x_2$ are in $C([0,1/2];\C)$, and $x_1(t)=x(t-1)$ and
$x_2(t)=x(t-1/2)$ for $t$ in the interval $[0,1/2]$. The time-$1/2$
map for the IVP of the ODE \eqref{eq:exdde} with initial conditions
$x_1(0)=v$, $x_2(0)=w$ is
\begin{displaymath}
  \begin{bmatrix}
    x_1(1/2)\\ x_2(1/2)
  \end{bmatrix}=
  \begin{bmatrix}
    \cosh\left(\frac{1}{2}\sqrt{\mu}\right) &
    \sqrt{\mu}\sinh\left(\frac{1}{2}\sqrt{\mu}\right)\\[1ex]
    \sinh\left(\frac{1}{2}\sqrt{\mu}\right)/\sqrt{\mu} &
    \cosh\left(\frac{1}{2}\sqrt{\mu}\right)
  \end{bmatrix}
  \begin{bmatrix}
    v\\[1ex] w
  \end{bmatrix}
\end{displaymath}
where $\sqrt{\smash[b]{\mu}}$ refers to the principal branch of the
complex square root (the set of solutions is the same for both
branches as all entries in the matrix are even). The continuity
condition $x_2(0)=x_1(1/2)$ in \eqref{eq:exdde} is satisfied if and
only if
\begin{displaymath}
  \left[1-\sqrt{\mu}\sinh\left(\frac{1}{2}\sqrt{\mu}\right)\right]w+
  \cosh\left(\frac{1}{2}\sqrt{\mu}\right) v=0\mbox{,}
\end{displaymath}
which has a unique solution for arbitrary $v$ if and only if
\begin{displaymath}
  0\neq 1-\sqrt{\mu}\sinh\left(\frac{1}{2}\sqrt{\mu}\right)\mbox{.}
\end{displaymath}
Thus, we see that for example \eqref{eq:cdde} the initial-value
problem \eqref{eq:ivp}--\eqref{eq:ic} cannot be solved for non-zero
$v$ whenever $\mu$ is a root of the function
\begin{displaymath}
  f(\mu)=1-\sqrt{\mu}\sinh\left(\frac{1}{2}\sqrt{\mu}\right)\mbox{.}
\end{displaymath}
This function $f$ is a globally defined real analytic function of
$\mu$ (the expression on the right-hand side is even in $\sqrt{\mu}$).
It has an infinite number of complex roots accumulating at
infinity. One of the roots is real: $f(0)=1$, and $f(\mu)\to-\infty$
for $\mu\to+\infty$ and $\mu\in\R$ (let's call it $\mu_0$:
$\mu_0\approx1.8535$). Consequently, example \eqref{eq:cdde} provides
a counterexample to Hypothesis~\ref{thm:sz}.  If one extends example
\eqref{eq:cdde} by a decoupled equation that has its Floquet
multiplier at $1/\mu_0$ (for example, $\dot y=-[\log\mu_0]y$) then it
also contradicts the hypothesis as stated in \cite{SSH06}.

As we have discussed in Section~\ref{sec:constr}, one finds that
$\mu=1/\lambda$ is a pole of $\Delta(\mu)$ if $\lambda$ is a non-zero
eigenvalue of the operator $[\id-M_1(0)]^{-1}L$ where $M_1$
and $L$ are defined in \eqref{eq:varm} and
\eqref{eq:ldef}.  Thus, whenever $[\id-M_1(0)]^{-1}L$ has a
non-zero spectral radius we must expect that $\Delta(\mu)$ has poles.

\section{The extended characteristic matrix}
\label{sec:fix}
The construction of $\Delta(\mu)$ suffers from the problem that the
poles of $\Delta(\mu)$ may coincide with the inverse of Floquet
multipliers of \eqref{eq:spdde} of interest. A simple extension of the
construction of $\Delta(\mu)$ permits one to push the poles to the
outside of a circle of any desired radius $R$.  Then the
characteristic matrix can be used to find all Floquet multipliers
outside of the ball of radius $1/R$.  We explain the extension in detail
for a single delay $\tau<1$ (the straightforward generalization to the
general case is relegated to Appendix~\ref{sec:gen}).  

The idea is based on the observation that the linear part of the
right-hand side in the integral equation formulation
\eqref{eq:varp}--\eqref{eq:varm} of initial-value problem
\eqref{eq:ivp}--\eqref{eq:ic} has a norm less than $1$ if the interval
is sufficiently short (thus, making the fixed point problem
\eqref{eq:varp}--\eqref{eq:varm} uniquely solvable).

Similar to a multiple shooting approach, we partition the interval
$[-1,0]$ into $k$ intervals of size $1/k$:
\begin{equation}\label{eq:jidef}
  J_i=\left[t_i,t_{i+1}\right)=
  \left[-1+\frac{i}{k},-1+\frac{i+1}{k}\right)\mbox{\ for $i=0,\ldots,k-1$.}  
\end{equation}
Then we write a system of $k$ coupled initial value problems for
a vector of $k$ initial values $(v_0,\ldots,v_{k-1})^T\in\C^{nk}$:
\begin{align}
  \label{eq:ivpext}
  \dot x(t)&=A(t)x(t)+
  B(t)
  \begin{cases}
        x(t-\tau) &\mbox{if $t\in[\tau-1,0]$,}\\
    \mu x(1+t-\tau) &\mbox{if  $t\in[-1,\tau-1)$,}
  \end{cases}\\
  \label{eq:icext}
  x(t_i)&=v_i\quad\mbox{for $i=0,\ldots,k-1$}
\end{align}
where $t\in[-1,0]$.  System \eqref{eq:ivpext}--\eqref{eq:icext}
corresponds to a coupled system of $k$ differential equations for the
$k$ functions $x\vert_{J_i}$ in the time intervals $J_i$.  More
precisely, $x$ is defined as a solution of the fixed point problem
\begin{align}
  \label{eq:varpext}
  x(t)&=v_s(t)+\int\limits_{a_k(t)}^tA(s)x(s)+B(s)
    \begin{cases}
      x(s-\tau) &\mbox{if $s\in[\tau-1,0]$,}\\
      \mu x(1+s-\tau) &\mbox{if  $s\in[-1,\tau-1)$}
    \end{cases}\d s\mbox{\ where}\\
  \label{eq:varpextv}
  v_s(t)&=v_i\mbox{\quad if $t\in J_i$, }\\
  \label{eq:varpexta}
  a_k(t)&=t_i\mbox{\quad if $t\in J_i$.}    
\end{align}
\begin{figure}[ht]
  \centering
  \includegraphics[scale=0.85]{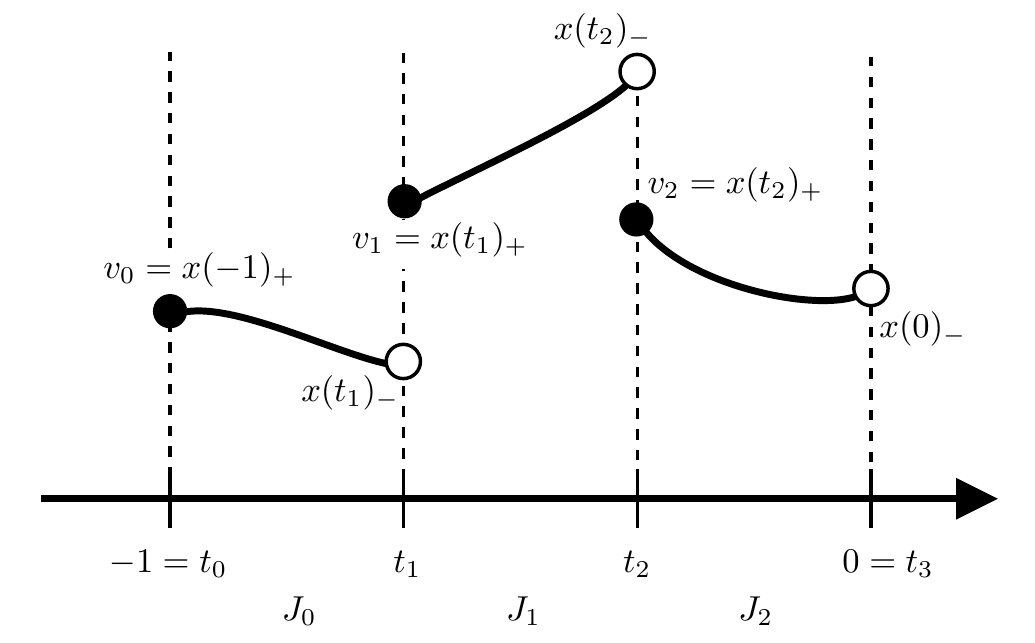}
  \caption{Illustration of solution $x\in C_k$ for multiple-shooting
    problem \eqref{eq:ivpext}--\eqref{eq:icext} with $k=3$
    sub-intervals. The right-sided limits $x(t_i)_+$ are equal to the
    variables $v_i$.  The gaps between the left-sided limits
    $x(t_i)_-$ and $x(t_i)_+$ are zero if $[v_0,\ldots,v_k]$ is in the
    kernel of the characteristic matrix $\Delta_k(\mu)$. The equations
    for $x\vert_{J_i}$ on the sub-intervals are coupled due to the
    terms $t-\tau$ or $1+t-\tau$ in \eqref{eq:ivpext}.}
  \label{fig:msproblem}
\end{figure}
We point out that a solution $x$ of
\eqref{eq:ivpext}--\eqref{eq:icext} is not necessarily in
$C([-1,0];\C^n)$ because it will typically have discontinuities at the
restarting times $t_i$ for $i=1\ldots k-1$ as illustrated in
Fig.~\ref{fig:msproblem} . Thus, we should define the space in which
we look for solutions as the space of \emph{piecewise continuous}
functions:
\begin{align}
  C_k=&\begin{aligned}[t]
  \{x:[-1,0]\mapsto\C^n:&\mbox{\ $x$ continuous on each
    (half-open)
    subinterval $J_i$}\\
  & \mbox{\ and $\lim_{t\nearrow t_i} x(t)$ exists for all $i=1\ldots k$.}\}    
  \end{aligned}\label{eq:ckdef}\\
  C_{k,0}=& \{x\in C_k: x(t_i)=0\mbox{\ for all $i=0\ldots
    k-1$}\}\mbox{.}\label{eq:ck0def}
\end{align}
We call $x\in C_k$ a solution of
\eqref{eq:varpext}--\eqref{eq:varpexta} if it satisfies
\eqref{eq:varpext} for every $t\in[-1,0]$. Both spaces, $C_k$ and
$C_{k,0}$ are equipped with the usual maximum norm.  For example, the
piecewise constant function $v_s$ as defined in \eqref{eq:varpextv} is
an element of $C_k$. Several operations are useful when dealing with
functions in $C_k$ to keep the notation compact.  Any element $x\in
C_k$ has well-defined one-sided limits at the boundaries $t_i$, which
we denote by the subscripts $+$ and $-$:
\begin{align*}
  x(t_i)_-&=\lim_{t\nearrow t_i} x(t)&&\mbox{for $i=1\ldots k$,}\\
  x(t_i)_+&=\lim_{t\searrow t_i} x(t)=x(t_i)&&\mbox{for $i=0\ldots
    k-1$.}
\end{align*}
We also define the four maps
\begin{equation}
  \label{eq:opsdef}
\begin{aligned}
  S&:\C^{nk}\mapsto C_k &&S[v_0\ldots v_{k-1}]^T(t)=v_i
  \mbox{\ if $t\in[t_i,t_{i+1})$ for $i=0\ldots k-1$,}\\
  \Gamma_+&:C_k\mapsto \C^{nk} &&
  \Gamma_+[x(\cdot)]=\left[x(-1)_+,x(t_1)_+,\ldots, 
    x(t_{k-1})_+\right]^T\mbox{,}\\
  \Gamma_-(\mu)&:C_k\mapsto \C^{nk} && 
  \Gamma_-(\mu)[x(\cdot)]= 
  \left[\mu x(0)_-,x(t_1)_-,\ldots, x(t_{k-1})_-\right]^T\mbox{,}\\
  M_k(\mu)&:C_k\mapsto C_{k,0} &&\mbox{\ where\ }\\
  &\!\!\!\left[M_k(\mu)x\right](t)=&&\int\limits_{a_k(t)}^tA(s)x(s)+B(s)    
  \begin{cases}
    x(s-\tau) &\mbox{if $s\in[\tau-1,0]$,}\\
    \mu x(1+s-\tau) &\mbox{if  $s\in[-1,\tau-1)$}
  \end{cases}\d s
\end{aligned}
\end{equation}
(the piecewise constant function $a_k$ was defined in
\eqref{eq:varpexta} as $a_k(t)=t_i$ if $t\in [t_i,t_{i+1})$). The map
$S$ takes a tuple of vectors $v_0,\ldots v_{k-1}$ and maps it to a
piecewise constant function, assigning $x(t_i)=v_i$ and then extending
with a constant to the subinterval $J_i$. For example, the function
$v_s$ in \eqref{eq:varpext}--\eqref{eq:varpextv} is equal to $Sv$. The
map $\Gamma_+$ takes the right-side limits of a piecewise continuous
function (thus, $\Gamma_+S$ is the identity in $\C^{nk}$). The map
$\Gamma_-(\mu)$ takes the left-side limits at the interior boundaries
$t_i$ ($i\geq1$), and in its first component it takes the left-side
limit at $t_k=0$ (the end of the interval) multiplied by $\mu$. The
map $M_k(\mu)$ is the generalization of $M_1$ defined in
\eqref{eq:varm}: in the definition of $M_k$ the lower boundary in the
integral is not $-1$ but $a_k(t)$ such that we can estimate its norm:
\begin{equation}
  \label{eq:mknorm}
  \|M_k(\mu)\|_\infty\leq\frac{1}{k}\left[\max_{t\in[-1,0]} \|A(t)\|_\infty+|\mu| 
    \max_{t\in[0,1]} \|B(t)\|_\infty\right]=:\frac{C_*(|\mu|)}{k}\mbox{.}
\end{equation}
According to the Theorem of Arzel{\'a}-Ascoli (see \cite{DS58a}, page
266) the operator $M_k(\mu)$ is also compact because it maps any
bounded set into a set with uniformly bounded Lipschitz
constants. Using the above notation we can write the differential
equation \eqref{eq:ivpext}--\eqref{eq:icext} (or, rather, the
corresponding integral equation \eqref{eq:varpext}) as a fixed point
problem in $C_k$. For a given $v=[v_0,\ldots,v_{k-1}]\in\C^{nk}$ and
$\mu\in\C$ we find a function $x\in C_k$ such that
\begin{equation}
  \label{eq:fixpp}
  x=Sv+M_k(\mu)x\mbox{.}
\end{equation}
Due to the norm estimate \eqref{eq:mknorm} on $M_k(\mu)$ we can state
the following fact about the existence of a fixed point $x$ of
\eqref{eq:fixpp} for a given $v$ and $\mu$.
\begin{lemma}\label{thm:uniquext}
  Let $R\geq1$ be given.  If we choose $k>C_*(R)$ then the
  system \eqref{eq:ivpext}--\eqref{eq:icext} has a unique solution
  $x\in C_k$ for all initial values $(v_0\ldots v_{k-1})^T$ and all
  $\mu$ satisfying $|\mu|<R$.
\end{lemma}
\begin{proof} The differential equation
  \eqref{eq:ivpext}--\eqref{eq:icext} is equivalent to fixed point
  problem \eqref{eq:fixpp}. The norm estimate \eqref{eq:mknorm}
  implies that $\id-M_k(\mu)$ is invertible such that the solution $x$
  of \eqref{eq:fixpp} is given by $[\id-M_k(\mu)]^{-1}Sv$.
\end{proof}

Lemma~\ref{thm:uniquext} implies that for $\mu$ satisfying
$C_*(|\mu|)<k$ the operator $\id-M_k(\mu)$ is an isomorphism on $C_k$.
\subsection{Extended characteristic matrix}
We can use Lemma~\ref{thm:uniquext} to define the \emph{extended}
characteristic matrix $\Delta_k(\mu)$ for $\mu$ satisfying
$C_*(|\mu|)<k$. This is a matrix in $\C^{(nk)\times(nk)}$ and is
defined as
\begin{equation}\label{eq:deltaedef}
  \Delta_k(\mu)v=\Delta(\mu)
\begin{bmatrix}
  v_0\\ \vdots\\ v_{k-1}  
\end{bmatrix}=
\begin{bmatrix}
  v_{0}-\mu x(0)\\
  v_1-x(t_1)_-\\
  \vdots\\
  v_{k-1}-x(t_{k-1})_-
\end{bmatrix}
\end{equation}
where $x=\left[\id-M_k(\mu)\right]^{-1}Sv\in C_k$ is the unique
solution of the coupled system of initial-value problems,
\eqref{eq:ivpext}--\eqref{eq:icext}, with initial condition
$v=(v_0,\ldots,v_{k-1})^T\in\C^{nk}$.  The first row in the right side
of definition \eqref{eq:deltaedef} corresponds to the boundary
condition \eqref{eq:bc}. The other $k-1$ rows, $v_i-x(t_i)_-$,
guarantee for $v\in\ker\Delta(\mu)$ that the right and the left limit
of $x$ at the restarting times $t_i$ agree, making the gaps shown in
Fig.~\ref{fig:msproblem} zero. Thus, $x=[\id-M_k(\mu)]Sv$ is not only
in $C_k$ but in $C([-1,0];\C^n)$ if $\Delta_k(\mu)v=0$ . What is more, if
$\Delta_k(\mu)v=0$ then the right-hand side of \eqref{eq:ivpext} for
the solution $x=[\id-M_k(\mu)]^{-1}Sv$ is continuous such that $x$ is
continuously differentiable and satisfies for every $t$ the
differential equation \eqref{eq:bvp}--\eqref{eq:bc}, which defines the
eigenvalue problem for $T$.

Let us list several useful facts about $C_k$, $C_{k,0}$,
$\Delta_k(\mu)$, $T$ and the quantities defined in \eqref{eq:opsdef}.
\begin{remunerate}
\item We can express $\Delta_k(\mu)$ using the maps defined in
  \eqref{eq:opsdef}:
  \begin{equation}\label{eq:deltaeexp}
    \Delta_k(\mu)=\id-\Gamma_-(\mu)\left[\id-M_k(\mu)\right]^{-1}S\mbox{.}
  \end{equation}
\item $C_k$ is isomorphic to $\C^{nk}\times C_{k,0}$:
  \begin{equation}
    \label{eq:isock}
    \begin{aligned}
      (v,\phi)&\in\C^{nk}\times C_{k,0}&&\mapsto Sv+\phi\in C_k\\
      \psi&\in C_k &&\mapsto (\Gamma_+\psi,\psi-S\Gamma_+\psi)\in
    \C^{nk}\times C_{k,0}\mbox{.}
    \end{aligned}
  \end{equation}
\item\label{item:dde} The differential equation
  \begin{equation}\label{eq:T0}
    \begin{split}
      \dot x(t)&=A(t)x+ B(t)
      \begin{cases}
        x(t-\tau) &\mbox{if $t\in[\tau-1,0]$,}\\
        0 &\mbox{if  $t\in[-1,\tau-1)$,}
      \end{cases}\\
      x(-1)_+&=0\\
      x(t_i)_+&=x(t_i)_-\quad\mbox{for $i=1,\ldots,k-1$}
    \end{split}
  \end{equation}
  has only the trivial solution (it is identical to
  \eqref{eq:ivp}--\eqref{eq:ic} for $\mu=0$). Its integral formulation
  is $x=S\Gamma_-(0)x+M_k(0)x$. Thus, the operator
  $\id-S\Gamma_-(0)-M_k(0):C_k\mapsto C([-1,0];\C^n)$ (which is of
  Fredholm index $0$) is invertible.
\item The time-$1$ map $T$ of the linear DDE $\dot
  x(t)=A(t)x(t)+B(t)x(t-\tau)$, defined in \eqref{eq:tdef} for initial
  history segments in $x\in C([-1,0];\C^n)$, can be easily extended to
  initial history segments in $C_k$. In fact, the extension of $T$ to
  $C_k$ can be defined by
  \begin{align}\label{eq:texp}\allowdisplaybreaks
    T&=\left[\id-S\Gamma_-(0)-M_k(0)\right]^{-1}
    \left[S(\Gamma_-(1)-\Gamma_-(0))+M_k(1)-M_k(0)\right]\mbox{,}
    \nonumber \intertext{which maps $C_k$ into $C([-1,0];\C^n)$, such
      that} \id-\mu T&=\left[\id-S\Gamma_-(0)-M_k(0)\right]^{-1}
    \left[\id-S\Gamma_-(\mu)-M_k(\mu)\right]\mbox{,}
  \end{align}
  which maps $C_k$ into $C_k$ (see Appendix~\ref{sec:check} for a
  detailed decomposition of expression \eqref{eq:texp}). Since the
  extension of $T$ maps into $C([-1,0];\C^n)$ (the original domain of
  the non-extended $T$) its spectrum is identical to the spectrum of
  the original $T$.
\end{remunerate}

Lemma~\ref{thm:uniquext} and the above facts permit us to apply the
general theory developed in \cite{KL92}, relating the extended
characteristic matrix $\Delta_k(\mu)\in\C^{(nk)\times(nk)}$ to
spectral properties of the time-$1$ map $T$: eigenvalues,
eigenvectors, and the length of their Jordan chains.
\begin{lemma}[Jordan chains]
  \label{thm:multiplicity}
  Assume that $0<|\mu_*|<R$ and $k>C_*(R)$ where $C_*(R)$ is defined
  by \eqref{eq:mknorm}. If $\Delta_k(\mu_*)$ is regular then $1/\mu_*$ is in
  the resolvent set of the monodromy operator $T$. If $\Delta_k(\mu_*)$
  is not regular then $\lambda_*=1/\mu_*$ is an eigenvalue of $T$. The
  Jordan chain structure for $\lambda_*$ as an eigenvalue of $T$ is
  also determined by $\Delta_k$\textnormal{:}
  \begin{remunerate}
  \item the dimension $l_0$ of $\ker\Delta_k(\mu_*)$ is the geometric
    multiplicity of $\lambda_*$,
  \item the order $p$ of the pole $\mu_*$ of $\Delta_k(\mu)^{-1}$ is the length
    of the longest Jordan chain associated to $\lambda_*$,
  \item the order of $\mu_*$ as a root of $\det\Delta_k(\mu)$ is the
    algebraic multiplicity of $\lambda_*$,
  \item the Jordan chains of length less than or equal to $p$ can be
    found as solutions $(y_0,\ldots, y_{p-1})$ (where $y_{p-1}$ is
    non-zero) of the linear system
    \begin{equation}\label{eq:jcderiv}
      \begin{split}
        0&=\Delta_k^{(0)}y_0\\
        0&=\Delta_k^{(1)}y_0+\Delta_k^{(0)}y_1\\
        0&=\frac{\Delta_k^{(2)}}{2}y_0+\Delta_k^{(1)}y_1+\Delta_k^{(0)}y_2\\
        \vdots&\\
        0&=\frac{\Delta_k^{(p-1)}}{(p-1)!}y_0+\ldots+\Delta_k^{(1)}y_{p-2}+
        \Delta_k^{(0)}y_{p-1}
      \end{split}
    \end{equation}
    where $\Delta_k^{(j)}$ is the $j$th derivative of
    $\Delta_k(\mu)$ in $\mu_*$ and
    $\Delta_k^{(0)}=\Delta_k(\mu_*)$.
  \end{remunerate}
\end{lemma}
System~\eqref{eq:jcderiv} is equivalent to the requirement
\begin{displaymath}
  \Delta_k(\mu)\left[y_0+(\mu-\mu_*)y_1+\ldots+
    (\mu-\mu_*)^{p-1}y_{p-1}\right]=
  O\left((\mu-\mu_*)^p\right)
\end{displaymath}
for all $\mu\approx \mu_*$ (this is, in fact, the general definition of
a Jordan chain for a nonlinear eigenvalue problem).  The vectors $y_j$
in the Jordan chains found as solutions of \eqref{eq:jcderiv} are
still vectors in $\C^{nk}$. We give the precise relation between the
vectors $y_j$ and the generalized eigenvectors of $\id-\mu T$ in
Equation~\eqref{eq:jctranslation} in the proof below.

\begin{proof}
The proof extends and modifies the
construction used by \cite{SSH06}, and makes the general theory of
\cite{KL92} applicable to time-periodic delay equations.

First, we give a brief summary of the relevant statement of \cite{KL92}. Let
$G:\Omega\mapsto{\cal L}(X_1;Y_1)$ and $H:\Omega\mapsto{\cal
  L}(X_2;Y_2)$ be two functions which are holomorphic on the open
subset $\Omega$ of the complex plane $\C$, and map into the space of
bounded linear operators from one Banach space $X_{1,2}$ into another
Banach space $Y_{1,2}$. The functions $G$ and $H$ are called
\emph{equivalent} if there exist two functions $E:\Omega\mapsto{\cal
  L}(X_2;X_1)$ and $F:\Omega\mapsto{\cal L}(Y_1;Y_2)$ (also
holomorphic on $\Omega$) whose values are isomorphisms such that
\begin{displaymath}
  H(\mu)=F(\mu)G(\mu)E(\mu)\mbox{\ for all $\mu\in\Omega$.}
\end{displaymath}
If two operators $G$ and $H$ are equivalent then every Jordan chain
$(u_0,\ldots,u_{p-1})$ of the eigenvalue problem $H(\mu)u=0$ for
$\mu_*$ corresponds to exactly one Jordan chain $(v_0,\ldots,v_{p-1})$
of the eigenvalue problem $G(\mu)v=0$ at $\mu_*$ via the isomorphism
\begin{equation}\label{eq:gen:jctranslation}
  \begin{split}
    v_0&=E_0u_0\\
    v_1&=E_0u_1+E_1u_0\\
    \vdots&\\
    v_{p-1}&=E_0u_{p-1}+\ldots +E_{p-1}u_0\mbox{,}
  \end{split}
\end{equation}
where the operators $E_j$ are the expansion coefficients of the
isomorphism $E$ in $\mu=\mu_*$:
\begin{equation}\label{eq:eexpand}
  E(\mu)=E_0+(\mu-\mu_*)E_1+\ldots
  +(\mu-\mu_*)^{p-1}E_{p-1}+O((\mu-\mu_*)^p)\mbox{.}
\end{equation}
We construct the equivalence initially only for the case of a single
delay $\tau<1$.  The necessary modifications for arbitrary delays can
be found in Appendix~\ref{sec:gen} (the underlying idea is identical
but there is more notational overhead). Also, we restrict ourselves
here to merely stating what the operators $H$, $F$, $G$ and $E$ are,
relegating the detailed checks and calculations to
Appendix~\ref{sec:check}.

The subset of permissible $\mu$, $\Omega\subset\C$, is the ball of
radius $R$ around $0$, where the initial-value problem
\eqref{eq:ivpext}--\eqref{eq:icext} has a unique solution.  In our
equivalence the operator functions $G$ and $H$ are
\begin{align*}
  G(\mu)&= \id-\mu T:&  X_1&=C_k && \mapsto & 
  Y_1&=X_1\mbox{,}\\
  H(\mu)&=
  \begin{pmatrix}
    \Delta_k(\mu) & 0\\ 0& \id
  \end{pmatrix}: & X_2&=\C^{nk}\times C_{k,0}&&\mapsto&
  Y_2&=X_2 \mbox{.}
\end{align*}
Note that the spaces $X_1$ and $X_2$ are isomorphic via relation
\eqref{eq:isock}. The eigenvalue problem $G(\mu)v=0$ is the eigenvalue
problem for $T$ (reformulated for $\mu=\lambda^{-1}$). The eigenvalue
problem $H(\mu)[v,\phi]=0$ is equivalent to $\Delta_k(\mu)v=0$,
$\phi=0$. Thus, establishing equivalence between $G$ and $H$ proves
Lemma~\ref{thm:multiplicity}.  We define the isomorphisms $E$ and $F$
as
\begin{equation}
  \label{eq:efdef}
  \begin{split}
  E(\mu)[v,\phi]&=\left[\id-M_k(\mu)\right]^{-1}\left[Sv+\phi\right]
  \\
  F(\mu)\psi&=
  \begin{bmatrix}
    F_1(\mu)\psi\\ F_2(\mu)\psi
  \end{bmatrix}\\
  &=
  \begin{bmatrix}
    \left[\Gamma_+-\Gamma_-(0)\right]\psi+
    \Gamma_-(\mu)\left[\id-M_k(\mu)\right]^{-1}
    \left[\id-S\Gamma_+-M_k(0)\right]\psi\\
    \left[\id-S\Gamma_+-M_k(0)\right]\psi
  \end{bmatrix}
\end{split}
\end{equation}
See \eqref{eq:deltaeexp} and \eqref{eq:texp} for expressions giving
$\Delta_k(\mu)$ and $\id-\mu T$, defining $G$ and $H$ in terms of $S$,
$\Gamma_\pm$ and $M_k$. The inverse of $\id-M_k(\mu)$ is guaranteed to
exist by Lemma~\ref{thm:uniquext}. The only choice we make is the
definition of $E$, which is clearly an isomorphism from $X_2$ to $X_1$
(because $(v,\phi)\mapsto Sv+\phi$ is an isomorphism from $X_2$ to
$X_1$). The definition of $F$ is then uniquely determined if we want
to achieve the equivalence $F(\mu)G(\mu)E(\mu)=H(\mu)$, and follows
from a straightforward but technical calculation, given in
Appendix~\ref{sec:check}. The concrete form of $E$ and $F$ allows us
to specify the precise relation between Jordan chains of
$\Delta_k(\mu)y=0$ and $(\id-\mu T)x=0$: $(x_0,\ldots,x_{p-1})\in
C([-1,0];\C^n)^p$ is a Jordan chain of $(\id-\mu T)x=0$ in $\mu=\mu_*$
if and only if
\begin{equation}\label{eq:jctranslation}
  \begin{aligned}
    x_0=&E_0[y_0,0]\mbox{,} &
    y_0=&\Gamma_+x_0\\
    x_1=&E_0[y_1,0]+E_1[y_0,0]\mbox{,} &
    y_1=&\Gamma_+x_1-\Gamma_+E_1[y_0,0]\\
    \vdots&&\vdots\\
    x_{p-1}=&E_0[y_{p-1},0]+\ldots+E_{p-1}[y_0,0]\mbox{,}  &
    y_{p-1}=&\Gamma_+x_{p-1}-\Gamma_+E_1[y_{p-2},0]-
    \ldots\\
    & &&
    -\Gamma_+E_{p-1}[y_0,0]
    \mbox{,}
  \end{aligned}
\end{equation}
where $(y_0,\ldots,y_{p-1})$ is a Jordan chain of $\Delta_k(\mu)$ in
$\mu=\mu_*$ and $E_j$ are the expansion coefficients of $E(\mu)$ in
$\mu_*$ (see \eqref{eq:eexpand}). In relation~\eqref{eq:jctranslation}
we used that the finite-dimensional component of $E(\mu)^{-1}x$ is
$\Gamma_+x$ (see \eqref{eq:einvdef} in the Appendix).
\end{proof}

\section{Conclusions}
\label{sec:conc}

We have generalized the construction proposed by Szalai et al.\
\cite{SSH06} such that it can be used to find a low-dimensional
characteristic matrix for any linear DDEs with time-periodic
coefficients. 

A shortcoming of our construction is that the function
$\det\Delta_k(\mu)$ does not depend smoothly on the delay $\tau$. This
problem is purely due to the particulars of our construction because
one expects that the roots of $\det\Delta_k(\cdot)$ depend smoothly on
$\tau$. One can remedy this problem by constructing a characteristic
matrix $\Delta_k(\cdot)$ that depends on the Floquet \emph{exponents}
instead of the Floquet multipliers. Then one can operate in the space
$C^k_\mathrm{per}([-1,0];\C^n)$ of $k$ times continuously
differentiable functions on $[-1,0]$ that are periodic in all
derivatives up to order $k$. One also has to choose a map $S:\C^{nk}$
that maps into $C^k_\mathrm{per}([-1,0];\C^n)$, which requires a
modification of the construction of the fixed point problem
\eqref{eq:varpext}--\eqref{eq:varpexta}.

An open question is how our extended matrix is related to the
characteristic matrix introduced in the textbook of Hale and
Verduyn-Lunel \cite{HL93} for periodic DDEs with delay identical to
the period. In particular, what happens to the poles of $\Delta(\mu)$
as the delay approaches $1$ (the period)?

\appendix

\section{Details of the equivalence in the proof of
  Lemma~\ref{thm:multiplicity}}
\label{sec:check}
\paragraph*{Inverse of $E(\mu)$}
Let $[v,\phi]=[(v_0,\ldots,v_{k-1})^T,\phi]\in\C^{nk}\times C_{k,0}$
be an element of the space $X_2$. Then the definition of
$x=E(\mu)[v,\phi]$ in \eqref{eq:efdef} means that $x$ is the solution
of the fixed point problem
\begin{equation}
  \label{eq:efixpp}
  x=Sv+\phi+M_k(\mu)x\mbox{.}
\end{equation}
Consequently, we can recover $v$ by applying $\Gamma_+$ to
\eqref{eq:efixpp} ($\Gamma_+\phi=0$ since $\phi\in C_{k,0}$ and
$\Gamma_+M_K(\mu)=0$ since $M_k$ maps into $C_{k,0}$). Then $\phi$ can
be recovered as $\phi=x-S\Gamma_+x-M_k(\mu)x$ such that the inverse of
$E$ is
\begin{equation}\label{eq:einvdef}
   E^{-1}(\mu)x=
   \begin{bmatrix}
     \Gamma_+x\\ [\id-S\Gamma_+-M_k(\mu)]x
   \end{bmatrix}\mbox{.}
\end{equation}

\paragraph*{Expression \eqref{eq:texp} for $\id-\mu T$}
Before finding an expression for $F(\mu)$ we check that
$G(\mu)=\id-\mu T$ is indeed given by the expression \eqref{eq:texp}.
If we denote the image of $x$ under $G(\mu)$ by $y$ then $y$ has the
form $x-\tilde y$ where $\tilde y=\mu Tx$. By definition
\eqref{eq:tdef} of the monodromy operator $T$, if $x\in
C([-1,0];\C^n)$ then $\tilde y$ is the unique solution of the
inhomogeneous initial-value problem on $[-1,0]$
\begin{align}
  &\begin{aligned}
    \Dot{\tilde{y}}(t)&= A(t)\tilde y(t)+B(t)
    \begin{cases}
      \tilde y(t-\tau)&\mbox{if  $t\in[\tau-1,0]$}\\
      \mu x(1+t-\tau) &\mbox{if $t\in[-1,\tau-1)$}
    \end{cases}\label{eq:tde}
  \end{aligned}\\
  &  \begin{aligned}
    \tilde y(-1)&=\mu x(0)\\
    \tilde y(t_i)+&=\tilde y(t_i)_-\mbox{\ for $i=1\ldots k-1$.}      
    \end{aligned}\label{eq:tic}
\end{align}
This solution $\tilde y$ is also in $C([-1,0];\C^n)$. We note that the
initial conditions in \eqref{eq:tic} mean that $\Gamma_+\tilde
y=\Gamma_-(0)\tilde y+[\Gamma_-(\mu)-\Gamma_-(0)]x$. Thus, the fixed
point formulation corresponding to \eqref{eq:tde}--\eqref{eq:tic} is
\begin{equation}
  \label{eq:tvoc}
  \tilde y=S\left[\Gamma_-(0)\tilde y+[\Gamma_-(\mu)-\Gamma_-(0)]x\right]+
  M_k(0)\tilde y+[M_k(\mu)-M_k(0)]x\mbox{,}
\end{equation}
which can be rearranged for $\tilde y$ because the factor
$\id-S\Gamma_-(0)-M_k(0)$ (which is in front of $y$) is invertible as
explained below \eqref{eq:T0} in the list of useful facts in
\S\ref{sec:fix}:
\begin{equation}
  \label{eq:muTx}
  \tilde y=\left[\id-S\Gamma_-(0)-M_k(0)\right]^{-1}
  \left[S\Gamma_-(\mu)-S\Gamma_-(0)+M_k(\mu)-M_k(0)\right]x\mbox{.}
\end{equation}
We notice that the operator in front of $x$ is also defined for all
$x\in C_k$, which allows us to extend the definition of $x\mapsto
\tilde y=\mu Tx$ to $x\in C_k$.  Moreover, when applying $\Gamma_+$ to
\eqref{eq:tvoc} we observe that $\tilde y$ is continuous in $[-1,0]$,
that is, the map $x\mapsto \mu Tx$ maps $C_k$ into
$C([-1,0];\C^n)$. Inserting $\tilde y=x-y$ into \eqref{eq:muTx} gives
expression \eqref{eq:texp}, which is equivalent to the fixed point
problem
\begin{equation}
  \label{eq:tfixpp}
  y=S\Gamma_-(0)y+M_k(0)y+x-S\Gamma_-(\mu)x-M_k(\mu)x\mbox{.}
\end{equation}

\paragraph*{Determination of $F(\mu)$}
From the fixed point equation \eqref{eq:tfixpp} it
becomes clear how to choose the other isomorphism $F(\mu)$ such that
$F(\mu)G(\mu)E(\mu)[v,\phi]=H(\mu)[v,\phi]$.  We observe that
\eqref{eq:tfixpp} implies (by applying $\Gamma_+$ to both sides)
\begin{equation}
  \label{eq:gygx}
    \Gamma_+y=\Gamma_-(0)y+\Gamma_+x-\Gamma_-(\mu)x\mbox{.}
\end{equation}
We apply $S$ to this identity and subtract it from the fixed point
equation \eqref{eq:tfixpp} defining $y$:
\begin{displaymath}
  y-S\Gamma_+y=M_k(0)y+x-S\Gamma_+x-M_k(\mu)x\mbox{.}
\end{displaymath}
If $x=E(\mu)[v,\phi]=[\id-M_k(\mu)]^{-1}[Sv+\phi]$ then
\eqref{eq:einvdef} implies $\Gamma_+x=v$, such that we obtain for
$y=G(\mu)x$
\begin{equation}\label{eq:f2def}
  y-S\Gamma_+y-M_k(0)y=\phi\mbox{.}
\end{equation}
Thus, $[\id-S\Gamma_+-M_k(0)]G(\mu)E(\mu)[v,\phi]=\phi$ for all
$(v,\phi)\in X_1$, which justifies our choice of the second component
of $F(\mu)$ in \eqref{eq:efdef}.  Inserting the expression
$E(\mu)[v,\phi]=[\id-M_k(\mu)]^{-1}[Sv+\phi]$ for $x$ and
$\Gamma_+x=v$ into \eqref{eq:gygx} we obtain that
\begin{align}
    \Gamma_+y-\Gamma_-(0)y&=v-\Gamma_-(\mu)[\id-M_k(\mu)]^{-1}[Sv+\phi]
    \label{eq:gyvphi1}\\
    &=[\id-\Gamma_-(\mu)[\id-M_k(\mu)]^{-1}S]v-
    \Gamma_-(\mu)[\id-M_k(\mu)]^{-1}\phi\nonumber\\
    &=\Delta_k(\mu)v- \Gamma_-(\mu)[\id-M_k(\mu)]^{-1}\phi\mbox{,}
    \label{eq:gyvphi2}
\end{align}
where we used the definition \eqref{eq:deltaeexp} of the
characteristic matrix $\Delta_k(\mu)$.  Since we know that $\phi$ can
be recovered from $y$ via \eqref{eq:f2def} we obtain from
\eqref{eq:gyvphi2} the first component of the definition of $F(\mu)$
in \eqref{eq:efdef}:
\begin{displaymath}
  \Gamma_+y-\Gamma_-(0)y+
  \Gamma_-(\mu)[\id-M_k(\mu)]^{-1}\left[\id-S\Gamma_+-M_k(0)\right]y=
  \Delta_k(\mu)v\mbox{.}
\end{displaymath}

\paragraph*{Inverse of $F(\mu)$}
It remains to be checked that $F$ is invertible if $|\mu|<R$. Given
$v\in\C^{nk}$ and $\phi\in C_{k,0}$ such that
\begin{align}
  v&=\Gamma_+y-\Gamma_-(0)y+
  \Gamma_-(\mu)[\id-M_k(\mu)]^{-1}\phi\mbox{,}\label{eq:vy}\\
  \phi&=y-S\Gamma_+y-M_k(0)y\mbox{,}\label{eq:phiy}
\end{align}
how do we recover $y$? Applying $S$ to \eqref{eq:vy} and adding the result to
\eqref{eq:phiy} gives
\begin{equation}
  \label{eq:vphiy}
  Sv+\phi=S\Gamma_-(\mu)[\id-M_k(\mu)]^{-1}\phi+
  \left[\id-S\Gamma_-(0)-M_k(0)\right]y\mbox{,}
\end{equation}
which is equivalent to \eqref{eq:vy}--\eqref{eq:phiy} because the map
$[v,\phi]\in \C^{nk}\times C_{k,0}\mapsto Sv+\phi\in C_k$ is an
isomorphism. Equation~\eqref{eq:vphiy} can be rearranged for $y$
because $\id-S\Gamma_-(0)-M_k(0)$ is invertible. Hence,
$[v,\phi]\in \C^{nk}\times C_{k,0}$ and $y\in C_k$ satisfy
$[v,\phi]=F(\mu)y$ (that is, system \eqref{eq:vy}--\eqref{eq:phiy}) if
and only if
\begin{displaymath}
  y=\left[\id-S\Gamma_-(0)-M_k(0)\right]^{-1}
  \left[Sv+\phi-S\Gamma_-(\mu)[\id-M_k(\mu)]^{-1}\phi\right]\mbox{.}
\end{displaymath}
Consequently,
\begin{equation}
  \label{eq:finvdef}
  F(\mu)^{-1}[v,\phi]= \left[\id-S\Gamma_-(0)-M_k(0)\right]^{-1}
  \left[Sv+\phi-S\Gamma_-(\mu)[\id-M_k(\mu)]^{-1}\phi\right]\mbox{.}
\end{equation}

\section{Characteristic matrices for general 
  periodic delay equations}
\label{sec:gen}
Consider a linear DDE with coefficients of time period $1$
and with maximal delay less than or equal to an integer  $m\geq 1$:
\begin{equation}
  \label{eq:genlong}
  \dot x(t)=\int_0^m\d_\theta\eta(t,\theta)x(t-\theta)
\end{equation}
where $\eta(t,\theta)$ is bounded, measurable and periodic in its
first argument $t$ (with period $1$), and of bounded variation in its
second argument $\theta$. We choose $\eta(t,\cdot)\in
NBV(\R;\C^{n\times n})$, that is, $\eta(t,\cdot)=0$ on $(-\infty,0]$,
$\eta(t,\cdot)$ is continuous from the right on $(0,m)$, and
$\eta(t,\cdot)$ is constant on $[m,\infty)$. In addition we assume
that the total variations of all $\eta(t,\cdot)$ have a common upper
bound:
\begin{equation}\label{eq:etabound}
  V_0^m\eta(t,\cdot)\leq \bar V
\end{equation}
where $V_0^mf$ is the total variation of $f$ \cite{DGLW95}.  In formulation
\eqref{eq:genlong} the dependent variable is the function $x$ on the
history interval $[-m,0]$.  The time-$1$ map $Tx$ of
\eqref{eq:genlong} for $x\in C([-m,0];\C^n)$ is defined as the
solution $y\in C([-m,0];\C^n)$ of the equation
\begin{align}\label{eq:tgende}
 t&>-1:& \dot y(t)&=\int_0^m\d_\theta\eta(t,\theta)
  \begin{cases}
    y(t-\theta) &\mbox{if $t-\theta\geq -1$}\\
    x(1+t-\theta)&\mbox{if $t-\theta< -1$}
  \end{cases}\\
  \label{eq:tgenshift}
  t&\leq-1: & y(t)&=x(1+t)\mbox{.}
\end{align}
Only the part of $y$ on the interval $(-1,0]$ is the result of a
differential equation (namely \eqref{eq:tgende}). The remainder of
$y$, $y\vert_{[-m,-1]}$, is merely a shift of $x$, given by
\eqref{eq:tgenshift}. Note that we have included the initial condition
for \eqref{eq:tgende} into \eqref{eq:tgenshift}, namely
$y(-1)=x(0)$. Correspondingly, the eigenvalue problem for $T$ reads
\begin{align}\label{eq:tevgende}
 t&>-1:& \dot x(t)&=\int_0^m\d_\theta\eta(t,\theta)
  \begin{cases}
    x(t-\theta) &\mbox{if $t-\theta\geq -1$}\\
    \mu x(1+t-\theta)&\mbox{if $t-\theta< -1$}
  \end{cases}\\
  \label{eq:tevgenshift}
  t&\leq-1: & x(t)&=\mu x(1+t)\mbox{.}
\end{align}
The spaces $C_k$ and $C_{k,0}$, and the maps $S$, $\Gamma_\pm$ and
$M_k(\mu)$ defined in \eqref{eq:ckdef}, \eqref{eq:ck0def} and
\eqref{eq:opsdef} for a single delay can be extended in a
straightforward way. We partition the interval $[-m,0]$ into
sub-intervals of length $1/k$, $J_i=[t_i,t_{i+1})=[-1+i/k,-1+(i+1)/k)$
($i=k-mk\ldots k-1$; note that $i$ can become negative), and define
the spaces
\begin{equation}
  \label{eq:ckgendef}
  \begin{split}
    C_k=&\begin{aligned}[t]
  \{x:[-m,0]\mapsto\C^n:&\mbox{\ $x$ continuous on all
    (half-open)
    subintervals $J_i$,}\\
  & \mbox{\ and $\lim_{t\nearrow t_i} x(t)$ exists for all $i=k-mk+1\ldots k$}\}    
  \end{aligned}\\
  C_{k,0}=&
  \{x\in C_k: x(t_i)=0\mbox{\ for  all $i=0\ldots k-1$}\}\mbox{.}  
\end{split}
\end{equation}
Note that non-negative indices $i$ of $t_i$ correspond to times $t_i$
in the interval $[-1,0]$ whereas negative indices $i$ correspond to
times $t_i$ in the interval $[-m,-1)$. These points of discontinuity
are treated differently in the definition of $C_{k,0}$: to be an
element of $C_{k,0}$ the function $x$ is only required to be zero at
$t_i$ with non-negative $i$. The space $C_k$ consists of piecewise
continuous functions where the points of discontinuity are the times
$t_i\in[-m,0]$. For functions in $C_k$ we define the maps (again
analogous to the definitions in \eqref{eq:opsdef} for the single
delay)
\begin{equation}
  \label{eq:opsgendef}
  \begin{aligned}
    S&: \C^{nk}\mapsto C_k &&S[v_0\ldots v_{k-1}]^T(t)=
    \begin{cases}
      v_i & \mbox{if $t\in[t_i,t_{i+1})$ for $i=0\ldots k-1$,}\\
      0 &\mbox{if $t\in[-m,-1)$,}
    \end{cases}\\
    \Gamma_+&:C_k \mapsto \C^{nk} &&
    \Gamma_+x=\left[x(-1)_+,x(t_1)_+,\ldots,
      x(t_{k-1})_+\right]^T\mbox{,}\\
    \Gamma_-(\mu)&:C_k\mapsto \C^{nk} && \Gamma_-(\mu)x=
    \left[\mu x(0)_-,x(t_1)_-,\ldots, x(t_{k-1})_-\right]^T\mbox{,}\\
    M_k(\mu)&:C_k\mapsto C_{k,0} &&\mbox{\ where}\\
    &\left[M_k(\mu)x\right](t)=&&
  \begin{cases}
    \int\limits_{a_k(t)}^t\int\limits_0^m\d_\theta\eta(s,\theta)
    \begin{cases}
      x(s-\theta) &\mbox{if $s-\theta\geq -1$}\\
      \mu x(1+s-\theta)&\mbox{if $s-\theta< -1$}
    \end{cases} \\ &\hspace*{-5ex}\mbox{for $t>-1$,}\\
    \mu x(1+t) &\hspace*{-5ex}\mbox{for $t\leq-1$,}
  \end{cases}
  \end{aligned}
\end{equation}
and $a_k(t)=t_i$ if $t\in J_i$ and $i\geq0$.  The map $S$ extends a
tuple of vectors into a piecewise constant function using the elements
of the tuple as the values on the subintervals $J_i\subset[-1,0]$, and
setting the function to $0$ on $[-m,-1)$. $\Gamma_+$ and $\Gamma_-$
are defined in exactly the same way as in \eqref{eq:opsdef} (again,
the notation $x(t_i)_\pm$ refers to left-sided and right-sided limits
of $x$ at $t_i$). The definition of $M_k(\mu)$ is identical to the
single delay case for $t>-1$. For $t\leq-1$ we define $M_k(\mu)$ as a
shift multiplied by $\mu$. In the same manner as for the single delay
we consider the extended initial-value problem
\begin{align}
  \label{eq:ivpxgende}
 t&>-1:& \dot x(t)&=\int_0^m\d_\theta\eta(t,\theta)
  \begin{cases}
    x(t-\theta) &\mbox{if $t-\theta\geq -1$,}\\
    \mu x(1+t-\theta)&\mbox{if $t-\theta< -1$,}
  \end{cases}\\
  \label{eq:ivpxgenic}
  &&x(t_i)_+&=v_i \mbox{\quad for $i=0\ldots k-1$,}\\
  \label{eq:ivpxgenshift}
  t&\leq-1: & x(t)&=\mu x(1+t)\mbox{.}
\end{align}
Note that the initial values $v_i$ are given only at the discontinuity
points $t_i\in[-1,0]$ ($i\geq0$). The solution $x\in[-m,-1)$ is
defined by a \emph{backward} shift using \eqref{eq:ivpxgenshift}. The
precise meaning of the initial-value problem is given by its
corresponding integral equation, which is a fixed point problem in
$C_k$, and can be expressed as
\begin{equation}
  \label{eq:fixppgen}
  x=Sv+M_k(\mu)x
\end{equation}
using the maps $S$ and $M_k(\mu)$. The following lemma gives a
sufficient condition for the existence of a unique solution $x\in C_k$
of \eqref{eq:fixppgen} and the invertibility of $\id-M_k(\mu)$. It is
a straightforward generalization of Lemma~\ref{thm:uniquext}. The only
difference in the proof is that we can achieve $\|M_k(\mu)\|<1$ only
in a weighted maximum norm $\|\cdot\|_R$ (which is nevertheless
equivalent to $\|\cdot\|_\infty$).
\begin{lemma}[Unique solution for initial-value problem]
  \label{thm:uniquegen} Let $R>1$ be fixed. Then the fixed point
  problem \eqref{eq:fixppgen} has a unique solution $x$ for all
  $v\in\C^{nk}$ and for all complex $\mu$ satisfying $|\mu|<R$ if the
  number $k$ of the sub-intervals is chosen such that
  \begin{equation}\label{eq:genkbound}
    R^{1/k}<1+\bar V^{-1}\frac{R-|\mu|}{R^{m+1}}\log R\mbox{.}    
  \end{equation}
\end{lemma}
Note that $\bar V$ is an upper bound on the total variation of
$\eta$. The left-hand side of \eqref{eq:genkbound} approaches $1$ for
$k\to\infty$ whereas the right-hand side is larger than $1$ if $R>1$
and $R>|\mu|$.
\begin{proof}
  We define the norm
  \begin{displaymath}
    \|x\|_R=\max_{t\in[-m,0]}|R^t x(t)|
  \end{displaymath}
  on the space $\Lint^\infty([-m,0];\C^n)$ of bounded functions and the spaces
  $C_k$ and $C_{k,0}$ (which are subspaces of $\Lint^\infty$). This
  norm is equivalent to the standard maximum norm. Thus,
  \eqref{eq:fixppgen} has a unique solution in $C_k$ whenever
  $\|M_k(\mu)\|_R<1$.  Two operators that appear in $M_k(\mu)$ are shift and integration. Their respective domains and norms are:
  \begin{align*}
    S(\theta):&\ \Lint^\infty([-m,0];\C^n)\mapsto\Lint^\infty([-m,0];\C^n) 
    &&\mbox{defined by}\\
    \left[S(\theta)x\right](t)&=
    \begin{cases}
      x(t-\theta) &\mbox{if $t-\theta\in[-m,0]$,}\\
      0 &\mbox{otherwise,}
    \end{cases}
    &&\|S(\theta)x\|_R\leq R^\theta\|x\|_R\\
    J:&\Lint^\infty([-m,0];\C^n)\mapsto C_{k,0} &&\mbox{defined by}\\
    [Jx](t)&=\int\limits_{a_k(t)}^t x(s)\d s\mbox{,} 
    &&\|Jx\|_R\leq
    \frac{R^{1/k}-1}{\log R}\|x\|_R\mbox{.}
  \end{align*}
  In order to find the $\|\cdot\|_R$-norm of the integrand (which is
  in $\Lint^\infty([-m,0];\C^n)$ if we extend it with zero for
  $t\in[-m,-1)$) in the definition of $M_k(\mu)$,
  \eqref{eq:opsgendef}, we have to estimate the $\|\cdot\|_R$-norm for
  a Stieltjes sum over an arbitrary partition
  $0\leq\theta_0<\ldots<\theta_N\leq m$ (with arbitrary intermediate
  values $\tilde\theta_j\in[\theta_j,\theta_{j+1}]$):
  \begin{align*}
    \lefteqn{\left\|\sum_{j=0}^{N-1}
      \left[\eta(t,\theta_{j+1})-\eta(t,\theta_j)\right]
      \begin{cases}
        x(t-\tilde\theta_j) &\mbox{if $t-\tilde\theta_j\geq -1$}\\
        \mu x(1+t-\tilde\theta_j)&\mbox{if $t-\tilde\theta_j< -1$}
      \end{cases}\right\|_R}\\
  \leq&\sup_{t\in[-1,0]}R^t\sum_{j=0}^{N-1}
      \left|\eta(t,\theta_{j+1})-\eta(t,\theta_j)\right|
      \begin{cases}
        \left|\left[S(\tilde\theta_j)x\right](t)\right| 
        &\mbox{if $t-\tilde\theta_j\geq -1$}\\
        R\left|\left[S(\tilde\theta_j-1)x\right](t)\right|
        &\mbox{if $t-\tilde\theta_j< -1$}
      \end{cases}\\
  \leq&\sup_{t\in[-1,0]}\sum_{j=0}^{N-1}
      \left|\eta(t,\theta_{j+1})-\eta(t,\theta_j)\right|
      \begin{cases}
        \phantom{R\,}\|S(\tilde\theta_j)x\|_R &\mbox{if $t-\tilde\theta_j\geq -1$}\\
        R\,\|S(\tilde\theta_j-1)x\|_R&\mbox{if $t-\tilde\theta_j< -1$}
      \end{cases}\\
  \leq&\sup_{t\in[-1,0]}\sum_{j=0}^{N-1}
      \left|\eta(t,\theta_{j+1})-\eta(t,\theta_j)\right|
      R^{\tilde\theta_j}\|x\|_R\\
      \leq& R^m\|x\|_R
    \sup_{t\in[-1,0]}V_0^m\eta(t,\cdot)= R^m\|x\|_R\bar V\mbox{.}
  \end{align*}
  Combined with the norm estimate for integration from $a_k(t)$ to $t$
  we obtain that
  \begin{equation}\label{eq:genmkbound01}
    \|\left[M_k(\mu)x\right]\vert_{[-1,0]}\|_R\leq
    \|J\|_RR^m\bar V\|x\|_R\leq 
    R^m\bar
    V \frac{R^{1/k}-1}{\log R}\|x\|_R\mbox{.}    
  \end{equation}
  The part of $M_k(\mu)x$ on the interval $[-m,-1)$ has the
  $\|\cdot\|_R$-norm:
  \begin{equation}\label{eq:genmkbound1m}
    \|\left[M_k(\mu)x\right]\vert_{[-m,-1)}\|_R\leq |\mu|\|S(-1)x\|_R\leq 
    |\mu|R^{-1}\|x\|_R\mbox{,}
  \end{equation}
  where we note that $S(-1)$ maps $C_k$ back into $C_k$.
  Adding up the inequalities \eqref{eq:genmkbound01} and
  \eqref{eq:genmkbound1m} we obtain an upper bound for
  $\|M_k(\mu)x\|_R$:
  \begin{equation}
    \label{eq:proofgenmkbound}
    \|M_k(\mu)x\|_R\leq \left[R^m\bar
      V \frac{R^{1/k}-1}{\log R}+\frac{|\mu|}{R}\right]\|x\|_R\mbox{.}
  \end{equation}
  The factor in front of $\|x\|_R$ on the right-hand side is less than
  unity if $k$ satisfies the inequality \eqref{eq:genkbound} required
  in the statement of the lemma.
\end{proof}

The invertibility of $\id-M_k(\mu)$ permits us to choose exactly the
same constructions for the characteristic matrix $\Delta_k(\mu)$ and
for the isomorphisms $E$ and $F$ proving Lemma~\ref{thm:multiplicity}
in the same way as for the single delay case (see
\eqref{eq:deltaeexp}, \eqref{eq:texp} and \eqref{eq:efdef}):
\begin{align*}
  X_1&=C_k\mbox{,}\\
  X_2&=\C^{nk}\times C_{k,0}\mbox{,}\\
  G(\mu)&=\id-\mu T=\left[\id-S\Gamma_-(0)-M_k(0)\right]^{-1}
  \left[\id-S\Gamma_-(\mu)M_k(\mu)\right]\mbox{,}\\
  H(\mu)&=\begin{pmatrix}
    \Delta_k(\mu) & 0\\ 0& \id
  \end{pmatrix}\mbox{,}\\
  \Delta_k(\mu)&=\id -\Gamma_-(\mu)\left[\id-M_k(\mu)\right]^{-1}S
  \quad\in\C^{nk}\times\C^{nk}\mbox{,}\\
  E(\mu)&=\left[\id-M_k(\mu)\right]^{-1}[Sv+\phi]\mbox{,}\\
  F(\mu)&=  \begin{bmatrix}
    \Gamma_+-\Gamma_-(0)+
    \Gamma_-(\mu)\left[\id-M_k(\mu)\right]^{-1}
    \left[\id-S\Gamma_+-M_k(0)\right]\mbox{,}\\
    \id-S\Gamma_+-M_k(0)\mbox{.}
  \end{bmatrix}
\end{align*}
With these constructions the maps $E(\mu):X_2\mapsto X_1$ and
$F(\mu):X_1\mapsto X_2$ are isomorphisms for $|\mu|<R$ (see
\eqref{eq:einvdef} and \eqref{eq:finvdef} for the inverses). The
relation $H(\mu)=F(\mu)G(\mu)E(\mu)$ makes the infinite-dimensional
eigenvalue problem of the time-$1$ map, $[\id-\mu T]x=0$, equivalent
to the finite-dimensional eigenvalue problem $\Delta_k(\mu)v=0$.

We remark that the inverse of $\id-S\Gamma_-(0)-M_k(0)$, which is
present in the definition of $G(\mu)$, exists because the
initial-value problem \eqref{eq:tgende}--\eqref{eq:tgenshift} has only
the trivial solution for $x=0$ (the operator is a compact perturbation
of the identity). In other words, the existence of this inverse
follows from the fact that the time-$1$ map $Tx$ is well defined (and
equal to zero) in $x=0$. Moreover, the extension of $T$ to $C_k$, defined by
\begin{displaymath}
  T=\left[\id-S\Gamma_-(0)-M_k(0)\right]^{-1}
  \left[S(\Gamma_-(1)-\Gamma_-(0))+M_k(1)-M_k(0)\right]\mbox{,}
\end{displaymath}
which maps $C_k$ back into $C_k$, has the property that $T^mx$ is
continuous (that is, $T^mx\in C([-m,0];\C^n)$) for all $x\in
C_k$. Hence, the spectrum of the extension $T:C_k\mapsto C_k$ is the
same as the spectrum of the original $T:C([-m,0];\C^n)\mapsto
C([-m,0];\C^n)$.

\bibliographystyle{siam} \bibliography{delay}

\end{document}